\documentclass[reqno,11pt]{amsart}
\usepackage{latexsym,amsmath,amsthm,amssymb,amsxtra,mathrsfs}
\usepackage[all]{xy}
\usepackage{times}
\usepackage[none]{hyphenat}
\usepackage[bookmarks]{hyperref}
\hypersetup{%
    naturalnames,%
    bookmarksnumbered,%
    plainpages=false,%
    pdfstartview=FitH,%
    pdfborderstyle={/S/U/W 1},%
    pdfdisplaydoctitle=true,%
    pdftitle={Purely Algebraic Method to Construct Toric Schemes},%
    pdfauthor={Ting Li},%
    pdfsubject={Algebraic Geometry},%
    pdfkeywords={Monoid, Fan, Toric Variety, Toric Scheme}%
}

\DeclareSymbolFont{txfontsC}{U}{txsyc}{m}{n}
\SetSymbolFont{txfontsC}{bold}{U}{txsyc}{bx}{n}
\DeclareFontSubstitution{U}{txsyc}{m}{n}

\let\notin\undefined
\DeclareMathSymbol{\notin}{\mathrel}{txfontsC}{60}
\DeclareMathSymbol{\defeq}{\mathrel}{txfontsC}{66}

\renewcommand\emptyset{\varnothing}
\renewcommand{\sim}{\thicksim}

\newdir{ (}{{}*!/-1ex/@^{(}}

\numberwithin{equation}{section}
\let\osphat=\sphat
\renewcommand{\sphat}{\osphat\ }

\newcommand{\Groth}[1]{\ensuremath{{#1}^{\mathrm{gp}}}}
\newcommand{\defset}[2]{\{\, #1 \mid #2 \,\}}
\newcommand{\lhom}[2]{\mathrm{Hom}(#1 ,\, #2)}
\newcommand{\abs}[1]{\lvert{#1}\rvert}
\DeclareMathOperator{\Spec}{Spec}

\theoremstyle{definition}
\newtheorem{monoD}{Definition}[section]
\theoremstyle{plain}
\newtheorem{monoT}[monoD]{Theorem}
\newtheorem{monoL}[monoD]{Lemma}

\newcommand{\Temb}[1][R]{\ensuremath{\mathrm{T}(\triangle,#1)}}
\newcommand{\Memb}[1][R]{\ensuremath{\mathscr{M}(\triangle,#1)}}
\newcommand{\Temba}[1][R]{\ensuremath{\mathrm{T}(\triangle',#1)}}
\newcommand{\Memba}[1][R]{\ensuremath{\mathscr{M}(\triangle',#1)}}

\begin{document}

\title{Purely Algebraic Method to Construct Toric Schemes}
\author{Ting Li}
\address{Department of Mathematics\\
    Peking University\\
    Beijing 100871\\
    P. R. China}
\email{etale.moduli@gmail.com}
\keywords{Monoid, Fan, Toric Variety, Toric Scheme}

\begin{abstract}
In this article, we first give some elementary proprieties of monoids and fans, then construct a toric scheme over an arbitrary ring, from a given fan. Using Valuative Criterion, we prove that this scheme is separated and give the sufficient and necessary condition when it is proper. We also study the regularity and logarithmic regularity of it. Finally we study the morphisms of toric schemes induced by the homomorphisms of fans.
\end{abstract}

\maketitle

\section*{Introduction}
The study of toric varieties began in the early 1970's. Till now, people have made great progress in this field. Many applications in algebraic geometry, complex manifolds and number theory etc.~have been discovered. Tadao Oda \cite{TOda} is a comprehensive book on this area. Till now the treatment of toric varieties need more or less the technique of complex analysis. This may hamper its applications in number theory, where we must study algebraic scheme over any base ring, especially over $\mathbb{Z}$ or finite fields. So I am determined to find a purely algebraic approach once and for all.

The keystone of this kind of construction is the concepts of \emph{convex cone} and \emph{concave cone}. In geometry, convex cones are always contained in some affine space, whereas convex cones here are contained in some free abelian group of finite rank, which acts the role of the affine space. For convenience to algebraic treatment, we invent its dual notion \emph{concave cone}. Having these two concepts as bricks, we can easily define fans and toric schemes (over arbitrary rings); and their properties, for example separatedness and properness, easily follow.

We also construct logarithmical structures on toric schemes and prove that when their base rings are regular, they are logarithmically regular. Hence by \cite{KKato}, their singularities can be solved.

In this article, rings, algebras and monoids are all assumed to be commutative and have identity elements. A homomorphism of rings (resp.~monoids) is assumed to preserve the identity element. A subring (resp.~monoid) is assumed to contain $1$ of the total ring (resp.~monoid).

\section{Monoids and Ideals}
In this section, we list some elementary results about monoids and ideals, which can also be found in many other references, for example, \cite{KKato}.

\begin{monoL}
A monoid $M$ is integral if and only if $M \to \Groth{M}$ is injective, where $\Groth{M}$ is the Grothendieck group associated with $M$.
\end{monoL}

\begin{monoD}
Let $M$ be an integral monoid and $P$ a submonoid of $M$. If $\Groth{P}\cap M=P$, then we say that $P$ is \emph{full} in $M$.
\end{monoD}

\begin{monoL}
Let $M$ be an integral monoid and $S$ a submonoid of $M$.
\begin{enumerate}
\item $S^{-1}M$ is an integral monoid with Grothendieck group \Groth{M}.
\item If $M$ is saturated, so is $S^{-1}M$.
\end{enumerate}
\end{monoL}

\begin{monoL}\label{ds4646}
Let $M$ be an integral monoid and $N$ a submonoid of $M$. Then
\begin{enumerate}
\item $M/N$ is also integral.
\item There is an isomorphism of groups:
\[\Groth{(M/N)} \stackrel{\sim}{\rightarrow} \Groth{M}/\Groth{N},\quad \frac{\bar{x}}{\bar{y}} \mapsto \overline{\frac{x}{y}}\,.\]
\item If $M$ is saturated, so is $M/N$.
\end{enumerate}
\end{monoL}

\begin{monoL}\label{dsf5798}
Let $M$ be a finitely generated saturated monoid. Then there is a submonoid $N$ of
$M$ such that $M=N \times M^{\ast}$. Moreover all such submonoids $N$ satisfy
the following conditions.
\begin{enumerate}
\item $N^{\ast}=\{1\}$ and $\Groth{M}=\Groth{N}\times M^{\ast}$.
\item $N$ is finitely generated and saturated.
\end{enumerate}
\end{monoL}

\begin{monoD}
Let $M$ be a monoid. A subset $I$ of $M$ is called an \emph{ideal} of $M$ if
$MI \subseteq I$. An ideal $I$ of $M$ is called a \emph{prime ideal} if its complement
$M-I$ is a submonoid of $M$. We denote by $\Spec M$ the set of prime ideals of $M$.

For example, $\emptyset$ and $M-M^{\ast}$ are prime ideals of $M$.

For $\mathfrak{p} \in \Spec M$, we define
$M_{\mathfrak{p}} \defeq S^{-1}M$, where $S=M-\mathfrak{p}$.

For a homomorphism $h \colon M \to N$ of monoids, we have a
canonical map
$\Spec N \to \Spec M,\ \mathfrak{p} \mapsto h^{-1}(\mathfrak{p})$.
\end{monoD}

\begin{monoT} \label{iuyiu7}
Let $M$ be a monoid and $\mathfrak{p}$ a prime ideal of $M$. Put $N=M-\mathfrak{p}$. Then:
\begin{itemize}
\item[(1)] $\mathfrak{q} \mapsto N^{-1}\mathfrak{q}$ is a bijective from the set
$\defset{\mathfrak{q} \in \Spec M}{\mathfrak{q} \subseteq \mathfrak{p}}$ to
$\Spec M_{\mathfrak{p}}$.
\item[(2)] $\mathfrak{P} \mapsto \mathfrak{P} \cap N$ is a bijective from the set
$\defset{\mathfrak{P} \in \Spec M}{\mathfrak{p} \subseteq \mathfrak{P}}$
to $\Spec N$ with the inverse map
$\mathfrak{q} \mapsto \mathfrak{q} \cup \mathfrak{p}$.
\end{itemize}
Assume that $M$ is integral, then:
\begin{itemize}
\item[(3)] $(M_{\mathfrak{p}})^{\ast}=\Groth{N}$.
\item[(4)] There is a canonical isomorphism of monoids:
\[M/N \cong M_{\mathfrak{p}}/(M_{\mathfrak{p}})^{\ast}\,.\]
Hence $(M/N)^{\ast}=\{1\}$.
\item[(5)] If $M$ is finitely generated, so is $N$.
\end{itemize}
Assume that $M$ is saturated, then:
\begin{itemize}
\item[(6)] $N$ is a saturated, full submonoid of $M$.
\item[(7)] $\Groth{M}/\Groth{N}$ is torsion free.
\end{itemize}
\end{monoT}

\begin{monoD}
For a monoid $M$, we define $\dim(M)$ to be the maximal length of
a sequence of prime ideals
${\mathfrak{p}}_0 \subset {\mathfrak{p}}_1 \subset \cdots \subset {\mathfrak{p}}_r$
of $M$. If such a maximum does not exist, we define $\dim(M) = \infty$.
\end{monoD}

\begin{monoL}\label{niy4f3}
Let $M$ be a finitely generated monoid. Then $\Spec M$ is a finite set.
\end{monoL}

\begin{monoL}\label{fdsrgf}
Let $k$ be a field and $M$ a finitely generated integral monoid.
Assume that $\Groth{M}$ is torsion free. Then
\[\dim (k[M]) = \mathrm{rank}(\Groth{M})\,.\]
\end{monoL}

\begin{monoL}\label{dsag46}
Let $M$ be a finitely generated integral monoid. Assume that $M^{\ast}=\{1\}$ and \Groth{M} is torsion free. Then $\dim(M)=1$ if and only if $M \cong \mathbb{N}$.
\end{monoL}

\section{Algebras Generated by Monoids}
\begin{monoL}\label{r7f65h}
Let $M$ be a finitely generated saturated monoid with $M^{\ast}=\{0\}$.
Then $M$ is isomorphic to a submonoid of ${\mathbb{N}}^n$ for
some integer $n \geqslant 1$.
\end{monoL}

\begin{proof}
For convenience to iterate, we view $M$ as an additive monoid. We use induction on
$r=\mathrm{rank}(\Groth{M})$. If $r=0$, then $M=\{0\}$; and if $r=1$,
$M \cong \mathbb{N}$.

Now assume that $r>1$. By Lemma \ref{dsag46}, $\dim(M)>1$. Let $\mathfrak{p}_1$ be a maximal element of the set $\defset{\mathfrak{p} \in \Spec M}{\mathfrak{p} \neq M-\{0\}}$. Put $N_1=M-\mathfrak{p}_1$. By Theorem \ref{iuyiu7}(2), $\dim N_1 =1$. Hence by Lemma \ref{dsag46}, $N_1 \cong \mathbb{N}$. Put $N_1=\mathbb{N}\cdot x$, where $x \in M$.

Let $k$ be a field. Put $A=k[M]$ and $\mathfrak{m}=(M-\{0\})\cdot A$.
By Lemma \ref{fdsrgf},
\[\mathrm{ht}(\mathfrak{m}) = \dim A = r > 1\,.\]
Let $\mathfrak{P}$ be a minimal prime ideal of $A$ containing $x$. Then
$\mathrm{ht}(\mathfrak{P})=1$. We have
\[x \in \mathfrak{p}_2 \defeq \mathfrak{P}\cap M \neq M-\{0\}\,.\]
Put $N_2=M-\mathfrak{p}_2$. By Theorem \ref{iuyiu7}, for each
$i=1,2$, $M_i/N_i$ is a finitely generated saturated monoid with
$(M_i/N_i)^{\ast}=\{0\}$. Obviously 
\[\mathrm{rank}\Groth{(M/N_i)} = \mathrm{rank}(\Groth{M}/\Groth{N_i})
<\mathrm{rank}(\Groth{M})=r\,.\]
By induction, there is an injective homomorphism
$l_i \colon M_i/N_i \to \mathbb{N}^{n_i}$ of monoids.
Let $\pi_i \colon M \to M/N_i$ be the canonical map. Put
\[\sigma \colon M \to M/N_1\oplus M/N_2,\quad x \mapsto (\pi_1(x),\pi_2(x))\,.\]

Let $\eta \in \Groth{N_1} \cap \Groth{N_2}$. Note that
$\Groth{N_1}=\mathbb{Z} \cdot x$. Hence if $\eta \neq 0$, we may
assume that $\eta=ax=y-z$, where $a>0$ and $y,\,z \in N_2$. Then
\[y=ax+z \in \mathfrak{p}_2\,,\]
a contradiction. So $\Groth{N_1} \cap \Groth{N_2} = \{0\}$, which infers that the
homomorphism
\[\Groth{\sigma} \colon \Groth{M} \to (\Groth{M}/\Groth{N_1})
\oplus(\Groth{M}/\Groth{N_2}),\quad \xi \mapsto (\Groth{\pi_1}(\xi),\Groth{\pi_2}(\xi))\]
is an injective, thus $\sigma$ is an injective. Put $n=n_1+n_2$. Then
\[l \defeq (l_1\oplus l_2) \circ \sigma \colon M \to \mathbb{N}^n\]
is an injective homomorphism of monoids.
\end{proof}

Applying \ref{r7f65h} and the results from \cite[p.320]{MHoch}, we obtain the following theorem.

\begin{monoT}\label{imbedtheorem}
Let $M$ be a monoid. Then the following conditions are equivalent:
\begin{enumerate}
\item $M$ is isomorphic to a submonoid of ${\mathbb{N}}^r$ for some integer $r>0$
and the ring $K[M]$ is integrally closed for some field $K$.
\item $M$ is finitely generated, saturated and $M^{\ast}=\{1\}$.
\item $M$ is isomorphic to a full submonoid of ${\mathbb{N}}^r$ for some
integer $r>0$.
\item $M$ is isomorphic to a finitely generated submonoid of ${\mathbb{N}}^r$
for some integer $r>0$ and for every integrally closed domain $D$, the ring
$D[M]$ is integrally closed.
\end{enumerate}
\end{monoT}

\section{Convex Cones and Concave Cones}
All abelian groups and monoid appearing in this section are presumed to be additive.

We fix a finitely generated free abelian group $G$. Put
\[G^{\star} \defeq \lhom{G}{\mathbb{Z}} \,.\]
We identify $G^{\star \star}$ with $G$.

\begin{monoD}
\
\begin{enumerate}
\item A finitely generated saturated submonoid $M$ of $G$ is called
a \emph{convex cone} in $G$ if $M^{\ast}=\{0\}$ and $G/\Groth{M}$
is torsion free.
\item A finitely generated saturated submonoid $N$ of $G$ is called a
\emph{concave cone} in $G$ if $\Groth{N}=G$ and $G/N^{\ast}$ is
torsion free.
\end{enumerate}
\end{monoD}

\begin{monoL} \label{gmonoid}
Let $M$ be a finitely generated saturated monoid, $\mathfrak{p}$ a prime ideal of $M$, $N=M-\mathfrak{p}$, $P$ a monoid and $b \in P$. Then there is a $\sigma \in \lhom{M}{P}$ such that $\sigma |_N=0$ and $\sigma(\mathfrak{p}) \subseteq b+P$.
\end{monoL}

\begin{proof}
Put $M_0=M/N$. Then $M_0$ is also a finitely generated saturated monoid such that
$M_0^{\ast}=\{0\}$ and $\Groth{M}_0$ is a free abelian group. By Theorem
\ref{imbedtheorem}, We may regard $M_0$ as a full submonoid of ${\mathbb{N}}^r$
for some integer $r>0$. Then there is a homomorphism $\tau \colon {\mathbb{N}}^r \to P$
of monoids such that
\[\tau(1,0,\ldots,0)=\tau(0,1,\ldots,0)=\cdots=\tau(0,0,\ldots,1)=b\,.\]
Put $\sigma=\tau|_{M_0}\circ \pi \colon M \to P$, where $\pi \colon M \to M_0$ is
the canonical homomorphism. Then $\sigma$ satisfies the required conditions.
\end{proof}

\begin{monoT}\label{dsy9797}
\
\begin{enumerate}
\item Let $M$ be a convex cone in $G$. Then
\[\check{M}\defeq\defset{\sigma \in G^{\star}}{\forall x \in M,\ \sigma(x)\geqslant 0}\]
is a concave cone in $G^{\star}$.
\item Let $N$ be a concave cone in $G$. Then
\[\hat{N}\defeq\defset{\sigma \in G^{\star}}{\forall x \in N,\ \sigma(x)\geqslant 0}\]
is a convex cone in $G^{\star}$.
\item For any convex cone $M$ in $G$, $(\check{M})\sphat=M$.
\item For any concave cone $N$ in $G$, $(\hat{N})\spcheck = N$.
\item Let $M$ be a convex cone in $G$ and $\mathfrak{p}\in \Spec M$.
Then $M-\mathfrak{p}$ is a convex cone in $G$.
\item Let $N$ be a concave cone in $G$ and $\mathfrak{q}\in \Spec N$.
Then $N_{\mathfrak{q}}$ is a concave cone in $G$. Moreover
$(N_{\mathfrak{q}})^{\ast}=\Groth{(N-\mathfrak{q})}$.
\end{enumerate}
\end{monoT}

\begin{proof}
(1) Obviously $\Groth{(\check{M})} \subseteq G^{\star}$ and $\check{M}$
is saturated. Assume that $M$ is generated by
$x_1,x_2,\ldots,x_r \in M-\{0\}$. 

Let $\sigma \in G^{\star}$. Then there is an $s \in \mathbb{N}$ such
that $s+\sigma(x_i)\geqslant 0$ for each $i$. By Lemma \ref{gmonoid},
there is a ${\tau}_1 \in \lhom{M}{\mathbb{N}}$ such that
\[{\tau}_1(M-\{0\}) \subseteq s+\mathbb{N}\,.\]
Since $G/\Groth{M}$ is torsion free,
$\Groth{{\tau}_1}\colon \Groth{M} \to \mathbb{Z}$ can be extended
to be a homomorphism $\tau \colon G \to \mathbb{Z}$. Obviously
$\tau,\,\tau+\sigma \in \check{M}$. Hence $\sigma \in \Groth{(\check{M})}$. Therefore
$G^{\star} = \Groth{(\check{M})}$.

It is easy to show that
\[(\check{M})^{\ast}=\defset{\sigma\in G^{\star}}{\Groth{M}\subseteq \ker(\sigma)}\,.\]
Hence we have
\[(\check{M})^{\ast} \cong (G/\Groth{M})^{\star} \cong G/\Groth{M}\]
is torsion free.

Note that
\[\omega \colon \check{M} \to {\mathbb{N}}^r,\quad
\sigma \mapsto (\sigma(x_1),\sigma(x_2),\ldots,\sigma(x_r))\]
is a homomorphism of monoids. Put $L=\mathrm{Im}(\omega)$ and let
$\xi=({\xi}_1,{\xi}_2,\ldots,{\xi}_r) \in \Groth{L}\cap{\mathbb{N}}^r$.
Write $\xi$ as $\xi = \omega(\sigma)-\omega(\tau)$, where
$\sigma,\,\tau \in \check{M}$. Then we have
\[(\sigma-\tau)(x_i) = \sigma(x_i)-\tau(x_i) ={\xi}_i \in \mathbb{N}\,.\]
Hence $\sigma-\tau \in \check{M}$, and $\xi = \omega(\sigma-\tau) \in L$.
Therefore $\Groth{L}\cap{\mathbb{N}}^r = L$, i.e., $L$ is a full
submonoid of ${\mathbb{N}}^r$. By Theorem \ref{imbedtheorem}, $L$
is finitely generated. So we may let
${\tau}_1,{\tau}_2,\ldots,{\tau}_m \in \check{M}$ such that $L$
is generated by $\omega({\tau}_1),\omega({\tau}_2),\ldots,\omega({\tau}_m)$.
${\tau}_1,{\tau}_2,\ldots,{\tau}_m \in \check{M}$ generated a submonoid
$L'$ of $\check{M}$. Now for any $\sigma \in \check{M}$,
there is a $\tau \in L'$ such that $\omega(\sigma)=\omega(\tau)$, i.e.,
$\sigma-\tau \in (\check{M})^{\ast}$. Hence
$\check{M} = L'+(\check{M})^{\ast}$ is finitely generated.
Therefore $\check{M}$ is a concave cone in $G^{\star}$.

(2) Obviously $\hat{N}$ is saturated and $(\hat{N})^{\ast}=\{0\}$.
Put $X=\defset{\sigma \in G^{\star}}{N^{\ast}\subseteq \ker \sigma}$.
Then $\Groth{(\hat{N})} \subseteq X$. Assume that $N$ is generated by
$x_1,x_2,\ldots,x_n$, where $x_1,x_2,\ldots,x_r \in N-N^{\ast}$
and $x_{r+1},\ldots,x_n \in N^{\ast}$.

Let $\sigma \in X$. Then there is an $s \in \mathbb{N}$ such that
$s+\sigma(x_i)\geqslant 0$ for each $1 \leqslant i \leqslant r$. By Lemma \ref{gmonoid},
there is a ${\tau}_1 \in \lhom{M}{\mathbb{N}}$ such that
\[{\tau}_1(N-N^{\ast}) \subseteq s+\mathbb{N}, \qquad {\tau}_1|_{N^{\ast}}=0 \,.\]
Put $\tau=\Groth{{\tau}_1}$. Then $\tau,\,\tau+\sigma \in \hat{N}$, i.e.,
$\sigma \in \Groth{\hat{N}}$. Therefore $\Groth{\hat{N}}=X$ and
$G^{\star}/\Groth{\hat{N}}$ is torsion free.

Obviously $\omega \colon \hat{N} \to {\mathbb{N}}^r$,
$\sigma \mapsto (\sigma(x_1),\sigma(x_2),\ldots,\sigma(x_r))$ is a
homomorphism of monoids. Since $\hat{N} \subseteq X$ and $\Groth{N} = G$,
$\omega$ is an injective. It is easy to show that $\omega(\hat{N})$
is a full submonoid of ${\mathbb{N}}^r$. Hence by Theorem \ref{imbedtheorem},
$\hat{N}$ is finitely generated.

Therefore $\hat{N}$ is a convex cone in $G^{\star}$.

(3) For any $x \in G$, we have
\[x \in(\check{M})\sphat \Longleftrightarrow \forall \sigma \in \check{M},\
\sigma(x)\geqslant 0\,.\]
Hence $M \subseteq (\check{M})\sphat$. On the other hand let $x \in G-M$. If
$x \notin \Groth{M}$, as $G/\Groth{M}$ is torsion free, there
is an element $\sigma \in G^{\ast}$ such that $\sigma|_{\Groth{M}}=0$ and
$\sigma(x)<0$, hence $x \notin (\check{M})\sphat$. Assume that
$x\in\Groth{M}$. By Theorem \ref{imbedtheorem}, we may regard $M$
as a full submonoid of ${\mathbb{N}}^r$ for some integer $r > 0$.
Let $p_i \colon {\mathbb{Z}}^r \to \mathbb{Z}$ be the $i$-th projection.
Since $\Groth{M} \cap {\mathbb{N}}^r = M$ and $x \notin M$,
there is a projection $p_i$ satisfying that $p_i(x)<0$.
$p_i|_{\Groth{M}} \colon \Groth{M} \to \mathbb{Z}$ can be extended
to a homomorphism $\sigma \in G^{\star}$. Then $\sigma \in \check{M}$
and $\sigma(x)<0$. Hence $x \notin (\check{M})\sphat$.
Therefore $M=(\check{M})\sphat$.
 
(4) For any $x \in G$, we have
\[x \in(\hat{N})\spcheck \Longleftrightarrow \forall \sigma \in \hat{N},\
\sigma(x)\geqslant 0\,.\]
Hence $N \subseteq (\hat{N})\spcheck$. Let $x \in G-N$. Let $\pi \colon G \to G/N^{\ast}$ be the canonical homomorphism. By Lemma \ref{ds4646}, we may regard that $N/N^{\ast} \subseteq G/N^{\ast}$ and $\Groth{(N/N^{\ast})} = G/N^{\ast}$. Obviously
$\pi(x) \notin N/N^{\ast}$ and $(N/N^{\ast})^{\ast}=\{0\}$. By Theorem \ref{imbedtheorem}, we may regard $N/N^{\ast}$ as a full submonoid of ${\mathbb{N}}^r$ for some integer $r > 0$. Let $p_i \colon {\mathbb{Z}}^r \to \mathbb{Z}$ be the $i$-th projection. Then there is a projection $p_i$ satisfying $p_i(\pi(x))<0$. Put $\sigma = p_i|_{G/N^{\ast}}\circ\pi$. Then $\sigma \in \hat{N}$ and $\sigma(x) < 0$. Hence $N = (\hat{N})\spcheck$. 

(5) and (6) are obvious.
\end{proof}

\begin{monoT}\label{cone1}
Let $M$ be a convex cone in $G$ and $N$ be a concave cone in $G$.
Let $\mathfrak{p}\in \Spec M$ and $\mathfrak{q}\in \Spec N$.
\begin{enumerate}
\item We have
\[\check{\mathfrak{p}}\defeq\defset{\sigma \in \check{M}}{\exists x \in M-\mathfrak{p},\ \sigma(x)>0} \in \Spec\check{M}\]
and
\[{\check{M}}_{\check{\mathfrak{p}}}=(M-\mathfrak{p})\spcheck.\]
\item We have
\[\hat{\mathfrak{q}}\defeq\defset{\sigma \in \hat{N}}{\exists x \in N-\mathfrak{q},\ \sigma(x)>0} \in \Spec\hat{N}\]
and
\[\hat{N}-\hat{\mathfrak{q}}=(N_{\mathfrak{q}})\sphat\ .\]
\item $(\check{\mathfrak{p}})\sphat=\mathfrak{p}$.
\item $(\hat{\mathfrak{q}})\spcheck = \mathfrak{q}$.
\item $\dim M=\dim\check{M}$ and $\dim N=\dim\hat{N}$.
\end{enumerate}
\end{monoT}

\begin{proof}
(1) Obviously $\check{\mathfrak{p}} \in \Spec \check{M}$ and
${\check{M}}_{\check{\mathfrak{p}}} \subseteq (M-\mathfrak{p})\spcheck$.
Assume that $M$ is generated by $x_1,x_2,\ldots,x_n$, where
$x_1,x_2,\ldots,x_r \in \mathfrak{p}$ and
$x_{r+1},\ldots,x_n \in M-\mathfrak{p}$. Let
$\sigma \in (M-\mathfrak{p})\spcheck$. Then there is an $s \in \mathbb{N}$
such that $s+\sigma(x_i) \geqslant 0$ for all $1 \leqslant i \leqslant r$. By
Lemma \ref{gmonoid}, there is a $\tau' \in \lhom{M}{\mathbb{N}}$
such that $\tau'|_{M-\mathfrak{p}}=0$ and
$\tau'(\mathfrak{p}) \subseteq s+\mathbb{N}$. 
$\Groth{\tau'} \colon \Groth{M} \to \mathbb{Z}$ can be extended
to a homomorphism $\tau \colon G \to \mathbb{Z}$. Then
$\tau \in \check{M}-\check{\mathfrak{p}}$ and $\tau+\sigma \in \check{M}$, i.e.,
$\sigma \in {\check{M}}_{\check{\mathfrak{p}}}$. Hence
${\check{M}}_{\check{\mathfrak{p}}} = (M-\mathfrak{p})\spcheck$.

(2) Obviously $\hat{\mathfrak{q}} \in \Spec \hat{N}$ and
$\hat{N}-\hat{\mathfrak{q}} \subseteq (N_{\mathfrak{q}})\sphat$.
By Theorem \ref{iuyiu7}(3), $(N_{\mathfrak{q}})^{\ast} = \Groth{(N-\mathfrak{q})}$.
Hence for all $\sigma \in (N_{\mathfrak{q}})\sphat$,
$\Groth{(N-\mathfrak{q})} \subseteq \ker \sigma$, i.e.,
$\sigma \in \hat{N}-\hat{\mathfrak{q}}$. Therefore
$\hat{N}-\hat{\mathfrak{q}} = (N_{\mathfrak{q}})\sphat$.

(3) For any $x \in G$,
\[x \in (\check{\mathfrak{p}})\sphat \Longleftrightarrow
x \in M\ {\rm and}\ \exists \sigma \in \check{M}-\check{\mathfrak{p}},\,\sigma(x)>0\,.\]
Hence $(\check{\mathfrak{p}})\sphat \subseteq \mathfrak{p}$.
By Lemma \ref{gmonoid}, there is a $\tau' \in \lhom{M}{\mathbb{N}}$
such that $\tau'|_{M-\mathfrak{p}}=0$ and
$\tau'(\mathfrak{p}) \subseteq 1+\mathbb{N}$.
$\Groth{\tau'} \colon \Groth{M} \to \mathbb{Z}$ can be extended
to a homomorphism $\tau \colon G \to \mathbb{Z}$.
Then $\tau \in \check{M}-\check{\mathfrak{p}}$ and $\tau(x)>0$
for any $x \in \mathfrak{p}$. Therefore
$(\check{\mathfrak{p}})\sphat = \mathfrak{p}$.

(4) For any $x \in G$,
\[x \in (\hat{\mathfrak{q}})\spcheck \Longleftrightarrow
x \in N\ {\rm and}\ \exists \sigma \in \hat{N}-\hat{\mathfrak{q}}
=(N_{\mathfrak{q}})\spcheck,\,\sigma(x)>0\,.\]
Since $(N_{\mathfrak{q}})^{\ast} = \Groth{(N-\mathfrak{q})}$,
$(\hat{\mathfrak{q}})\spcheck \subseteq \mathfrak{q}$. By Lemma
\ref{gmonoid}, there is a
${\tau}_1 \in \lhom{M}{\mathbb{N}}$ such that
${\tau}_1(\mathfrak{q}) \subseteq 1+\mathbb{N}$ and
${\tau}_1|_{N-\mathfrak{q}}=0$. Put
$\tau=\Groth{{\tau}_1} \colon G \to \mathbb{Z}$.
Then $\tau \in \hat{N}-\hat{\mathfrak{q}}$ and $\tau(x)>0$
for any $x \in \mathfrak{q}$. Therefore
$(\hat{\mathfrak{q}})\spcheck = \mathfrak{q}$.

(5) is from (3) and (4).
\end{proof}

\section{Toric Schemes}
In this and following sections, we need some knowledge of log structures.
\cite{FKato} and \cite{KKato} are good references in this field.

First we give the definition of \emph{fan}. Let $G$ be a finitely generated free abelian group.

\begin{monoD}
A \emph{fan} in $G$ is a nonempty collection $\triangle$ of
convex cones in $G$ satisfying the following condition:
\begin{itemize}
\item[(1)] for any $P\in \triangle$ and $\mathfrak{p} \in \Spec M$,
$P-\mathfrak{p} \in \triangle$.
\item[(2)] for any $P,\ Q \in \triangle,\ P \cap Q=P-\mathfrak{p}=Q-\mathfrak{q}$ for some $\mathfrak{p} \in \Spec P$
and $\mathfrak{q} \in \Spec Q$.
\end{itemize}
The union $\abs{\triangle}\defeq \bigcup_{P \in \triangle}P$ is called the \emph{support}
of $\triangle$. $\triangle$ is said to be \emph{finite} if $\triangle$ is a finite set, and is said
to be \emph{complete} if $\abs{\triangle}=G$.
\end{monoD}

Now we begin to construct the toric scheme from a given fan $\triangle$ in $G$. Let $R$ be a ring. For each $P \in \triangle$, let
\[\mathrm{U}_P\defeq\Spec R[\check{P}].\]
Then $\check{P} \to R[\check{P}]$ generates a log structure
on $\mathrm{U}_P$, denoted by $\mathscr{M}_P$.

Let $P \in \triangle$ and $\mathfrak{p}\in \Spec P$. Put
$Q=P-\mathfrak{p}$ and $S=\check{P}-\check{\mathfrak{p}}$. Then
$S$ is a finitely generated monoid. Since $\check{Q} =
{\check{P}}_{\check{\mathfrak{p}}}$ by Theorem \ref{cone1}(1), we
have
\[R[\check{Q}]=R[{\check{P}}_{\check{\mathfrak{p}}}]\cong S^{-1}R[\check{M}]\ ,\]
hence $\mathrm{U}_Q \to \mathrm{U}_P$ is an open immersion
and $\mathscr{M}_Q = \mathscr{M}_P|_{\mathrm{U}_Q}$.

Now let $P,\ Q\in \triangle$ and let $\mathfrak{p}\in \Spec P$, $\mathfrak{q}\in \Spec Q$
such that $P \cap Q = P-\mathfrak{p} = Q-\mathfrak{q}$. Then we may regard that
\[\mathrm{U}_P \cap \mathrm{U}_Q = \mathrm{U}_{P \cap Q}\ .\]
In this way we can glue $\defset{\mathrm{U}_P}{P\in \triangle}$ to form a scheme of
finite type over $R$, denoted by \Temb, which is called a \emph{toric scheme} over $R$.
Obviously $\defset{\mathscr{M}_P}{P\in\triangle}$ can be glued to form a log structure
on \Temb, denoted by \Memb.

Let $O\defeq\{0\}\in \triangle$. Then for any $P \in \triangle$,
$\mathrm{U}_O \subseteq \mathrm{U}_P$.

If $R = k$ is a field, by Lemma \ref{fdsrgf}, we have $\dim(\Temb[k]) = \mathrm{rank}(G)$.

\begin{monoL} \label{truyre657}
Let $P,Q \in\triangle$. Then the following conditions are equivalent:
\begin{enumerate}
\item $P \subseteq Q$.
\item $\mathrm{U}_P \subseteq \mathrm{U}_Q$.
\item There exists a point $x \in\mathrm{U}_P \cap \mathrm{U}_Q$
such that $\mathfrak{m}_{X,x}\cap\check{P}=\check{P}-\check{P}^{\ast}$.
\end{enumerate}
\end{monoL}

\begin{proof}
$(1)\Rightarrow (2)$ is by the definition of \Temb.

$(2)\Rightarrow (3)$. Put $A = R[\check{P}]$. Let $\mathfrak{m}$ be a maximal ideal
of $R$. Set $\mathfrak{P}=\mathfrak{m}\cdot A + (\check{P}-\check{P}^{\ast})\cdot A$.
Since $A/\mathfrak{P} \cong (R/\mathfrak{m})[\check{P}^{\ast}]$ is
integral, $\mathfrak{P} \in \Spec A$. $\mathfrak{P}$ correspond to
a point $x$ in $\mathrm{U}_P$, which satisfies the required conditions.

$(3) \Rightarrow (1)$. Let $\mathfrak{p}\in\Spec P$ such that
$P \cap Q = P-\mathfrak{p}$. Since $x \in \mathrm{U}_P\cap\mathrm{U}_Q =
\mathrm{U}_{P\cap Q} = \mathrm{U}_{P-\mathfrak{p}}$, we have
$\check{P}-\check{P}^{\ast} = \mathfrak{m}_{X,x}\cap\check{P} \subseteq \check{\mathfrak{p}}$.
Hence $\check{\mathfrak{p}}=\check{P}-\check{P}^{\ast}$ and
$\mathfrak{p} = (\check{\mathfrak{p}})\sphat = \emptyset$.
Therefore $P = P \cap Q \subseteq Q$.
\end{proof}

\begin{monoT}\label{t4654}
Let $R$ be a noetherian ring. Then the following conditions are equivalent.
\begin{enumerate}
\item \Temb\ is quasi-compact;
\item \Temb\ is a noetherian scheme;
\item $\triangle$ is finite.
\end{enumerate}
\end{monoT}

\begin{proof}
$(3) \Rightarrow (2) \Rightarrow (1)$ is obvious.

$(1) \Rightarrow (3)$. Since \Temb\ is quasi-compact, there are a finite
number of convex cones $P_1,P_2,\ldots,P_n$ in $\triangle$ such that
\[\Temb = \bigcup_{i=1}^{n} \mathrm{U}_{P_i}\ .\]
Put
\[\triangle'=\defset{P\in\triangle}{P \subseteq P_i\ \mbox{for some}\ i}\ .\]
By the definition of fan and Lemma \ref{niy4f3}, $\triangle'$ is a finite set.

Let $P \in \triangle$. Since $P \subseteq P$, by Lemma \ref{truyre657}, there is a point
$x \in \mathrm{U}_P$ such that
$\mathfrak{m}_{X,x}\cap \check{P} = \check{P}-\check{P}^{\ast}$.
Assume that $x \in \mathrm{U}_{P_i}$. Then by Lemma \ref{truyre657},
$P \subseteq P_i$, i.e., $P \in \triangle'$. Hence $\triangle = \triangle'$
is a finite set.
\end{proof}

\begin{monoT} \label{kfldsjf}
If $R$ is integral {\rm(}resp.~integrally closed{\,\rm)}, then \Temb\ is
integral {\rm(}resp.~normal{\,\rm)}.
\end{monoT}

\begin{proof}
Let $P \in \triangle$. By Lemma \ref{dsf5798}, there is a finitely generated saturated submonoid $N$ of $P$ such that $N^{\ast}=\{1\}$ and $\check{P} = N \times P^{\ast}$. We have $P^{\ast} \cong {\mathbb{Z}}^r$ for some $r \in \mathbb{N}$. Hence
\[R[\check{P}] \cong R[P^{\ast}][N] \cong R[{\mathbb{Z}}^r][N]\,.\]
By Theorem \ref{imbedtheorem} if $R$ is integral (resp.~integrally closed), so are $R[{\mathbb{Z}}^r]$ and $R[{\mathbb{Z}}^r][N]$.
\end{proof}

\begin{monoT}
Let $A$ be a $R$-algebra. Then
\[\Temb \times_R \Spec A \cong \Temb[A]\,.\]
\end{monoT}

In the following, we will use \emph{Valuative Criterion}(see [EGA] II 7.2.3 and 7.2.8) to discuss the separatedness and properness of toric schemes.

\begin{monoL} \label{t2313}
Let $K$ be a field, $\upsilon$ a valuation of $K$, $G$ a finitely generated
free abelian additive group, $f \colon G \to K^{\ast}$ a homomorphism of
abelian groups. Put $N\defeq\defset{x \in G}{\upsilon(f(x))\geqslant 0}$. If $G \neq N$, then
\begin{enumerate}
\item For any $x\in G$, $x\in N$ or $-x\in N$.
\item $N$ is a concave cone in $G$.
\item $\hat{N}\cong \mathbb{N}$, i.e., $\hat{N} = \mathbb{N} \cdot x$ for some $x \in G^{\star}$.
\end{enumerate}
\end{monoL}

\begin{proof}
(1) and (2) are obvious.

(3). Put $M=\hat{N}$. Then $M \neq \{0\}$ and $\check{M}=N$.
By Lemma \ref{dsag46}, We have only to prove that $\dim M=1$. Assume that
$\dim M \geqslant 2$. Then there is a $\mathfrak{p}\in\Spec M$ such
that $\mathfrak{p}\neq \emptyset,\, M-\{0\}$. Let $x\in
M-(\mathfrak{p} \cup \{0\})$ and $y \in \mathfrak{p}$. By Lemma
\ref{gmonoid}, there are $\sigma_1,\,\sigma_2\in\check{M}$ such
that $\sigma_1(x)>0,\,\sigma_2|_{M-\mathfrak{p}}=0$ and
$\sigma_2(y)>\sigma_1(y)$. Put $\sigma=\sigma_1-\sigma_2$. Then
$\sigma(x)>0$ and $\sigma(y)<0$. Hence
$\sigma,\,-\sigma \notin \check{M}=N$. This contradicts (1).
\end{proof}

\begin{monoL}\label{t4789}
Let $K$ be a field. Let $\varphi \colon \Spec K \to \Temb$ be a morphism
of schemes. Then $\varphi(\Spec K)\subseteq \mathrm{U}_O$.
\end{monoL}

\begin{monoL}\label{yusd4l}
\Temb\ is quasi-separated over $R$.
\end{monoL}

\begin{proof}
Note that for all $P,\ Q\in \triangle$, $\mathrm{U}_P \cap \mathrm{U}_Q = \mathrm{U}_{P \cap Q}$ is affine, a fortiori quasi-compact.
\end{proof}

\begin{monoT}\label{sep12}
\Temb\ is separated over $R$.
\end{monoT}

\begin{proof}
Let $(A,K,\upsilon)$ be a valuation ring. Let $\varphi \colon \Spec K \to \Temb$, $\psi \colon \Spec A \to \Spec R$ and $\phi_i \colon \Spec A \to \Temb (i=1,\,2)$ be morphisms of schemes which make a commutative diagram.
\begin{equation}\label{E:toric@s}
\vcenter{\xymatrix{\Spec K \ar[r]^{\varphi} \ar[d] & \Temb \ar[d]\\
\Spec A \ar@{-->}[ur]^{\phi_i} \ar[r]^{\psi} & \Spec R}}
\end{equation}
Let $\mathfrak{m}_{\upsilon}$ denote the maximal ideal of $A$. Let $P_i \in \triangle$
such that $\phi_i(\mathfrak{m}_{\upsilon})\in \mathrm{U}_{P_i}$. Then
$\phi_i(\Spec A)\subseteq \mathrm{U}_{P_i}$. By Lemma \ref{t4789},
$\varphi(\Spec K)\subseteq \mathrm{U}_O$. Hence Diagram \eqref{E:toric@s} induces a commutative diagram of rings.
\[\xymatrix{R \ar@{ (->}[r] \ar[d]^{l} &  R[\check{P_i}] \ar[ld]^{g_i} \ar[d]^{f_i} \ar@{ (->}[r] & R[G^{\star}] \\ A \ar@{ (->}[r] & K}\]
Put
\[N \defeq \defset{x \in G^{\star}}{\upsilon(f(x))\geqslant 0}\,.\] 
If $G^{\star} = N$, then $P \defeq \hat{N} = \{0\}$. If $G^{\star} \neq N$,
by the Lemma \ref{t2313}, $N$ is a concave in $G^{\star}$. Then
$\check{P_i}\subseteq N,\ \hat{N}\subseteq (\check{P_i})\sphat=P_i$.
We have
\[P \defeq P_1 \cap P_2 \supseteq \hat{N}\,.\]
In both cases, we have $P \in \triangle$ and $\check{P_i} \subseteq \check{P} \subseteq \check{N}$. Hence we have $f(R[\check{P_i}]) \subseteq f(R[\check{P}]) \subseteq A$ and $\phi_i(\Spec A) \subseteq \mathrm{U}_{\check{P}} \subseteq \mathrm{U}_{\check{P_i}}$. So $\phi_1=\phi_2$. As \Temb\ is quasi-separated, by Valuation Criterion, \Temb\ is separated over $R$.
\end{proof}

\begin{monoT}
\Temb\ is proper over $R$ if and only if $\triangle$ is a finite and complete fan.
\end{monoT}

\begin{proof}
(1) Assume that \Temb\ is proper over $R$. Let $\mathfrak{m}$ be a maximal ideal of
$R$ and set $k=R/\mathfrak{m}$. Then \Temb[k] is proper over $k$, thus is a noetherian
scheme. By Theorem \ref{t4654}, $\triangle$ is finite.

Let $x$ be any element in $G$. Then there is an element $y \in G$ such that $x=ry$
and $G/\mathbb{Z}y\cong \mathbb{Z}^{n-1}$, where $r\in\mathbb{N}$
and $n=\mathrm{rank}(G)$. Put $M \defeq \mathbb{N}y$, then
$M$ is a convex cone in $G$. Put
$N=\check{M},\ \mathfrak{p}=N-N^{\ast},\ A=k[N]$ and
$\mathfrak{P}=\mathfrak{p}A \in \Spec A$. Let $K$ be the quotient
field of $A$. Then there is a valuation ring $(B,\upsilon)$
of $K/k$ such that $A \subseteq B$ and
$\mathfrak{m}_{\upsilon} \cap A = \mathfrak{P}$. By Lemma \ref{t2313},
$N'\defeq\defset{u\in G^{\star}}{\upsilon(u)\geqslant 0}$ is a concave
cone in $G$ and $\hat{N'}\cong\mathbb{N}$. Since $N \subseteq N'$,
we have $\hat{N'} \subseteq \hat{N} = M$, hence $\hat{N'}=M$.
Since \Temb[k] is proper over $k$, $B$ has a unique center
$\xi$ on \Temb[k]. Let $P \in \triangle$ such that
$\xi \in \mathrm{U}_P$. Then $\check{P} \subseteq N'$ , and we have
$x \in M=\hat{N'} \subseteq P$. Hence $\triangle$ is complete.

(2) Assume that $\triangle$ is finite and complete. By Theorem \ref{sep12}, \Temb\ is separated over over $R$.

Let $(A,K,\upsilon)$ be a valuation ring. Let
$\varphi \colon \Spec K \to \Temb$ and $\psi \colon \Spec A \to \Spec R$
be morphisms of schemes which make a commutative diagram.
\begin{equation}\label{E:toric1}
\vcenter{\xymatrix{\Spec K \ar[r]^{\varphi} \ar[d] & \Temb \ar[d] \\
\Spec A \ar[r]^{\psi} & \Spec R}}
\end{equation}
By Lemma \ref{t4789}, $\varphi(\Spec K) \subseteq \mathrm{U}_O$. Hence \eqref{E:toric1} induces a diagram of rings.
\[\xymatrix{ R \ar[r]^{l} \ar@{ (->}[d] & A \ar@{ (->}[d] \\
  R[G^{\star}] \ar[r]^-{f} & K }\]
Put $N \defeq \defset{x \in G^{\star}}{\upsilon(f(x))\geqslant 0}$.
If $N=G^{\star}$, then $f(R[G^{\star}])\subseteq A$. Hence
$f \colon R[G^{\star}]\to A$ induces a morphism
$\phi \colon \Spec A \to \mathrm{U}_O \subseteq \Temb$ which make \eqref{E:toric7} commutative.
\begin{equation}\label{E:toric7}
\vcenter{\xymatrix{\Spec K \ar[r]^-{\varphi} \ar[d] & \Temb \ar[d] \\
\Spec A \ar[r]^{\psi} \ar@{-->}[ur]^{\phi} & \Spec R}}
\end{equation}
So we may assume that
$N \neq G^{\star}$. By Lemma \ref{t2313}, $N$ is a concave cone in $G^{\star}$
and $\hat{N} = \mathbb{N} \cdot \varepsilon$ for some
$\varepsilon \in G$. Since $\triangle$ is complete, there is a $P \in \triangle$ containing
$\varepsilon$. Then $\check{P} \subseteq N$. We have
$f(R[\check{P}]) \subseteq A$. Hence
\[g \defeq f|_{R[\check{P}]} \colon R[\check{P}] \to A\]
induces a morphism
\[\phi \colon \Spec A \to \mathrm{U}_P \subseteq \Temb\,,\]
which make \eqref{E:toric7} commutative. By Valuative Criterion, \Temb\ is proper over $R$.   
\end{proof}

\section{Regularity and Logarithmical Regularity}
Now we study the regularity of Toric Schemes.

\begin{monoL} \label{iuyj6d}
Let $R$ be a noetherian ring, $M$ a finitely generated saturated monoid such
that \Groth{M} is torsion-free. Then $R[M]$ is regular if and only if $R$ is regular and
$M \cong \mathbb{Z}^r \times \mathbb{N}^s$ for some $r,\, s \in \mathbb{N}$.
\end{monoL}

\begin{proof}
Put $A=R[M]$. Obviously if $R$ is regular and $M \cong \mathbb{Z}^r \times \mathbb{N}^s$, $A$ is regular.

Now assume that $A$ is regular. Let $\mathfrak{p}$ be any prime ideal of $R$. Then $\mathfrak{P} \defeq \mathfrak{p} \cdot A$ is a prime ideal of $A$. Since
$A$ is flat over $R$, $A_{\mathfrak{P}}$ is flat over $R_{\mathfrak{p}}$. By
\cite[Theorem 23.7]{HMat1}, $R_{\mathfrak{p}}$ is regular. Hence $R$ is a regular ring.

By Lemma \ref{dsf5798}, $M \cong \mathbb{Z}^r \times N$, where $N$ is a finitely generated saturated monoid with $N^{\ast} = \{1\}$. Put $A' = R[\mathbb{Z}^r]$. Let $\mathfrak{p}$ be a minimal prime ideal of $A'$. Since $R$ is regular, so is $A'$ and $A'_{\mathfrak{p}}$. Hence $K \defeq A'_{\mathfrak{p}}$ is a field. As $A'[N] \cong R[M]$ is regular, so is its localization $K[N] \cong S^{-1}(A'[N])$, where $S = A'-\mathfrak{p}$. Put $B=K[N]$ and $\mathfrak{m} = (N-\{1\}) \cdot B$. Set
\[T = N-\{1\}\cup \defset{x \cdot y}{x,\, y \in N-\{1\}}\ .\]
Then $N$ can be generated by $T$. Obviously $K \cdot T$ is a $K$-linear subspace
of $B$ with dimension $= |T|$. Put $\mathfrak{m}'=\mathfrak{m} B_{\mathfrak{m}}$.
Then we have
\[\mathfrak{m}'/ \mathfrak{m}'^2 \cong \mathfrak{m}/ \mathfrak{m}^2 \cong K \cdot T\ .\]
By Lemma \ref{fdsrgf}, $\dim B_{\mathfrak{m}'} = \mathrm{rank}(\Groth{N})$.
Since $B_{\mathfrak{m}'}$ is a regular local ring, we have
\[|T| = \dim_K(\mathfrak{m}' / \mathfrak{m}'^2) =  \dim B_{\mathfrak{m}'} = \mathrm{rank}(\Groth{N})\ .\]
Put $T = \{x_1,x_2,\ldots,x_s\}$, then the following homomorphism
\[\mathbb{Z}^s \to \Groth{N},\quad (a_1,a_2,\ldots,a_s) \mapsto x^{a_1}_1 x^{a_2}_2 \cdots x^{a_s}_s\]
is an isomorphism. Hence $N \cong \mathbb{N}^s$, i.e., $M \cong \mathbb{Z}^r \times \mathbb{N}^s$.
\end{proof}

\begin{monoT}
\Temb\ is a regular scheme if and only if $R$ is regular and for each $P \in \triangle$,
$P \cong \mathbb{N}^r$ for some $r \in \mathbb{N}$.
\end{monoT}

Next, we study the logarithmical regularity of Toric Schemes. First, we give the definition
of logarithmical regularity introduced in \cite[p.1075-1076]{KKato}.

\begin{monoD}\label{logdef1}
Let $(X,\mathscr{M})$ be a log scheme. We say $(X,\mathscr{M})$ is \emph{logarithmically regular} at a point $x \in X$, if the following two conditions are satisfied. Let $I(x,\mathscr{M})$ be the ideal of $\mathcal{O}_{X,x}$ generated by the image $\mathscr{M}_x - \mathcal{O}^{\ast}_{X,x}$.
\begin{enumerate}
\item $\mathcal{O}_{X,x}/I(x,\mathscr{M})$ is a regular local ring.
\item $\dim(\mathcal{O}_{X,x}) = \dim(\mathcal{O}_{X,x}/I(x,\mathscr{M})) + \mathrm{rank}(\Groth{\mathscr{M}}_x/\mathcal{O}^{\ast}_{X,x})$.
\end{enumerate}
We say $(X,\mathscr{M})$ is \emph{logarithmically regular} if $(X,\mathscr{M})$ is logarithmically regular at all $x \in X$.
\end{monoD}

\begin{monoL}\label{aff45n}
Let $R$ be a noetherian local ring with maximal ideal $\mathfrak{m}$, $M$ a finitely generated saturated monoid with $M^{\ast}=\{1\}$. Put $A=R[M]$ and $\mathfrak{m}'=\mathfrak{m}\cdot A+(M-\{1\})\cdot A$. Then $\mathfrak{m}'$ is a maximal ideal of $A$ with $\mathrm{ht}(\mathfrak{m}') = \dim(R)+\mathrm{rank}(\Groth{M})$.
\end{monoL}

\begin{proof}
Obviously $A$ is flat over $R$ and $\mathfrak{m}' \cap R =\mathfrak{m}$. We have
\begin{alignat*}{2}  \mathrm{ht}(\mathfrak{m}')&=\mathrm{ht}(\mathfrak{m})+\mathrm{ht}(\mathfrak{m}'/\mathfrak{m}A)& & \mbox{(by \cite{HMat}, (13.B), Theorem 19)}\\
&=\dim(R)+\dim(k[M])& &\quad (\mbox{put}\ k=R/\mathfrak{m}) \\
&=\dim(R)+\mathrm{rank}(\Groth{M})& &\quad \mbox{(by Lemma \ref{fdsrgf})} \qedhere
\end{alignat*}
\end{proof}

\begin{monoT}
Let $R$ be a regular ring. Then $(\Temb,\Memb)$ is logarithmically regular.
\end{monoT}

\begin{proof}
Put $X=\Temb$ and $\mathscr{M}=\Memb$. We use the notations introduced
in Definition \ref{logdef1}. Let $x$ be any point of $X$. Assume
that $x \in \mathrm{U}_P$, where $P\in\triangle$. Put $M=\check{P}$,
$A=\mathcal{O}_{X,x}$ and $I=I(x,\mathscr{M})$. Then
$\mathrm{U}_P = \Spec R[M]$ and $x$ corresponds to a prime ideal
$\mathfrak{P}$ of $R[M]$. Thus $A \cong R[M]_{\mathfrak{P}}$. Put
$\mathfrak{p}=M \cap \mathfrak{P}$. Then $I$ corresponds to the ideal
$\mathfrak{p}\cdot R[M]_{\mathfrak{P}}$ of $R[M]_{\mathfrak{P}}$.
By Lemma \ref{dsf5798}, there is an integer $r \geqslant 0$ and a finitely generated saturated monoid $N$ such that $N^{\ast}=\{1\}$ and $M_{\mathfrak{p}} \cong N \times \mathbb{Z}^r$. Then
\begin{alignat*}{2}
  A&\cong (R[M_{\mathfrak{p}}])_{\mathfrak{P}_1}\\
   &\cong (R_1[N])_{\mathfrak{P}_2}& &\quad (\mbox{put}\ R_1=R[\mathbb{Z}^r]) \\
   &\cong (R_2[N])_{\mathfrak{P}_3}& &\quad (\mbox{put}\ R_2=(R_1)_{\mathfrak{P}_2 \cap R_1})
\end{alignat*}
Obviously $R_2$ is a regular local ring with maximal ideal
$\mathfrak{m}_2=(\mathfrak{P}_2 \cap R_1)\cdot R_2$. Put
$A'=R_2[N]$, $\mathfrak{m}'=\mathfrak{m}_2\cdot A'+(N-\{1\})\cdot A'$
and $B=A'_{\mathfrak{P}_3}$. Since $\mathfrak{m}' \subseteq \mathfrak{P}_3$
and $\mathfrak{m}'$ is a maximal ideal of $A'$, we have
$\mathfrak{P}_3 = \mathfrak{m}'$. Note that $I$ corresponds to the
ideal $(N-\{1\})\cdot B$ of $B$. Hence $A/I \cong R_2$ is a regular
local ring. We have
\begin{alignat*}{2}
  \dim(A)&=\mathrm{ht}(\mathfrak{m}') \\
         &=\dim(R_2)+\mathrm{rank}(\Groth{N})& &\quad (\mbox{by Lemma \ref{aff45n}}) \\
         &=\dim(A/I)+\mathrm{rank}(\Groth{\mathscr{M}}_x/\mathscr{M}_x^{\ast})
\end{alignat*}
By the definition, $(X,\mathscr{M})$ is regular at $x$.
\end{proof}

\section{Morphisms of Toric Schemes}
\begin{monoD}
Let $(G,\triangle)$ and $(G',\triangle')$ be two fans. A homomorphism $\varphi \colon(G,\triangle)\to(G',\triangle')$ of fans is a homomorphism $\varphi \colon G \to G'$ of groups satisfying that: for each $P \in \triangle$ there exists a $P' \in \triangle'$ such that $\varphi(P) \subseteq P'$.
\end{monoD}

In the following, we let $R$ be a ring.

\begin{monoT}
Let $\varphi\colon(G,\triangle)\to(G',\triangle')$ be a homomorphism
of fans. Then $\varphi$ gives rise to a morphism
\[\varphi_{\ast} \colon (\Temb,\Memb) \to (\Temba,\Memba)\]
of log schemes over $R$.
\end{monoT}

\begin{monoT} \label{mnjuuw}
Let $\varphi\colon(G,\triangle)\to(G',\triangle')$ be a homomorphism
of fans. Then for any $P'\in\triangle'$,
\[\varphi_{\ast}^{-1}(\mathrm{U}_{P'}) = \bigcup_{\substack{P\in\triangle\\ \varphi(P)\subseteq P'}} \mathrm{U}_P\ .\]
\end{monoT}

\begin{proof}
Put $X=\Temb,\,X'=\Temba$ and $f=\varphi_{\ast}\colon X \to X'$. Let
$x \in f^{-1}(\mathrm{U}_{P'})$. We may assume that $x \in \mathrm{U}_{P_1}$,
where $P_1\in\triangle$. Then there is a $P''\in\triangle'$ such that
$\varphi(P_1)\subseteq P''$. We have $\mathfrak{p}'\in\Spec P'$ and
$\mathfrak{p}''\in\Spec P''$ such that
\[P' \cap P'' = P'-\mathfrak{p}' = P''-\mathfrak{p}''\ .\]
Put $Q_1=\Check{P_1},\ Q'=\Check{P'},\ Q''=\Check{P''},\ \mathfrak{q}'=\Check{\mathfrak{p}'}$
and $\mathfrak{q}''=\Check{\mathfrak{p}''}$. $x$ corresponds to a
prime ideal $\mathfrak{P}$ of $R[Q_1]$. Put $\mathfrak{q}_1=\mathfrak{P}\cap Q_1$,
$\psi=\varphi|_{P_1}\colon P_1 \to P''$ and $\phi=\varphi^{\star}|_{Q''}\colon Q'' \to Q_1$.
Since
\[\varphi(x) \in \mathrm{U}_{P'}\cap\mathrm{U}_{P''} = \mathrm{U}_{P'\cap P''} = \mathrm{U}_{P''-\mathfrak{p}''}\,,\]
we have a natural homomorphism of $R$-algebras: $R[Q''_{\mathfrak{q}''}] \to R[Q_1]_{\mathfrak{P}}$.
Using the following commutative diagram
\[\xymatrix{Q'' \ar@{^{(}->}[r] \ar[d]^{\phi} & Q''_{\mathfrak{q}''} \ar@{^{(}->}[r] & R[Q''_{\mathfrak{q}''}] \ar[d]\\
Q_1 \ar[rr] & & R[Q_1]_{\mathfrak{P}} }\]
we obtains $\phi^{-1}(\mathfrak{q}_1) \subseteq \mathfrak{q}''$.
Hence $\mathfrak{p} \defeq \psi^{-1}(\mathfrak{q}'') \subseteq \Hat{\mathfrak{q}_1}$.
Put $P=P_1 - \mathfrak{p}$. Then $\varphi(P) \subseteq P''-\mathfrak{p}'' \subseteq P'$.
Since $\mathfrak{q}_1 \subseteq \check{\mathfrak{p}}$,
$x \in \mathrm{U}_P \subseteq f^{-1}(U_{P'})$.
\end{proof}

\begin{monoT}
Let $\varphi \colon(G,\triangle)\to(G',\triangle')$ be a homomorphism of fans. Then $\varphi_{\ast}\colon \Temb \to \Temba$ is proper if and only if for each $P' \in \triangle'$, the set
\[\triangle_{P'} \defeq \defset{P \in \triangle}{\varphi(P)\subseteq P'}\]
is finite and
\[\varphi^{-1}(P') = \bigcap_{P \in \triangle_{P'}} P\ .\]
\end{monoT}

\begin{proof}
By Theorem \ref{mnjuuw} and Lemma \ref{niy4f3}, we may assume that $\triangle'$ is finite.

(1) Assume that $\varphi_{\ast}$ is a proper morphism. Let $\mathfrak{m}$ be a maximal ideal of $R$ and set $k=R/\mathfrak{m}$. Then
\begin{align*}
X & \defeq \Temb[k] \cong \Temb \times_R \Spec k \,,\\
X' & \defeq \Temba[k] \cong \Temba \times_R \Spec k \,,
\end{align*}
and the morphism
\[f \defeq \varphi_{\ast}\times_R\mathrm{id} \colon X \to X'\]
is proper. As $X'$ is a noetherian scheme by  Theorem \ref{t4654}, so is $X$. Hence $\triangle$ is a finite set.

Let $P' \in \triangle'$ and $u \in \varphi^{-1}(P')$ be any elements. Then there is an $e \in G$ such that $u=ae$ and $G/\mathbb{Z}e \cong \mathbb{Z}^{n-1}$, where $a \in \mathbb{N}$ and $n=\mathrm{rank}(G)$. $M \defeq \mathbb{N}e$ is a convex cone in $G$. Put $N=\check{M},\ \mathfrak{p}=N-N^{\ast},\ A=k[N]$ and $\mathfrak{P}=\mathfrak{p}A \in \Spec A$. Let $K$ be the quotient field of $A$. Then there is a valuation ring $(B,\upsilon)$ of $K/k$ such that $A \subseteq B$ and $\mathfrak{m}_{\upsilon} \cap A = \mathfrak{P}$. So $N = \defset{u \in G^{\star}}{\upsilon(u)\geqslant 0}$. As $P'$ is saturated and $\varphi(u)=a \cdot \varphi(e) \in P'$, we have $\varphi(e)\in P'$, hence $\varphi(M) \subseteq P'$. $\varphi$ induces a homomorphism $k[\Check{P'}] \to k[N]$, and the following commutative diagram of rings
\[\xymatrix{k[\Check{P'}] \ar[r] & k[N] \ar@{^{(}->}[r] \ar@{^{(}->}[d] & B \ar@{^{(}->}[d]\\ & k[G^{\star}] \ar@{^{(}->}[r] & K}\] 
induces a commutative diagram of schemes
\[\xymatrix{\Spec K \ar[r] \ar[d] & \mathrm{U}_O \ar@{^{(}->}[r] & X  \ar[d]^{f} \\
\Spec B \ar[r] & \mathrm{U}_{P'} \ar@{^{(}->}[r] & X' }\]
By Valuative Criterion, there is a morphism $g\colon \Spec B \to X$
of schemes which make a diagram
\[\xymatrix{\Spec K \ar[d] \ar[r] & X \ar[d]^{f}\\
\Spec B \ar[r] \ar@{-->}[ur]^{g} & X' }\]
Since $f(g(m_{\upsilon})) \in \mathrm{U}_{P'}$, by Theorem
\ref{mnjuuw}, there exists a
$P\in\triangle$ such that $g(m_{\upsilon})\in\mathrm{U}_P$ and
$\varphi(P) \subseteq P'$. As $k[\check{P}] \subseteq B$, we have $\check{P} \subseteq N$. Hence $u \in M = \hat{N} \subseteq P$.

(2) Assume that for each $P' \in \triangle'$, $\triangle_{P'}$ is finite and
\[\varphi^{-1}(P') = \bigcap_{P \in \triangle_{P'}} P\ .\]
Let $X \defeq \Temb$, $X' \defeq \Temba$ and $f = \varphi_{\ast} \colon X \to Y$.
By Theorem \ref{sep12}, $f$ is separated. Let $(A, K, \upsilon)$ be a valuation ring. Let $\alpha \colon \Spec K \to X$ and $\beta \colon \Spec A \to X'$ be the morphisms of schemes which make a commutative diagram.
\begin{equation}\label{E:map1}
\vcenter{\xymatrix{\Spec K \ar[r]^-{\alpha} \ar[d] & X \ar[d]^{f}\\
\Spec A \ar[r]^-{\beta} & X' }}
\end{equation}
Assume that $\beta(\mathfrak{m}_{\upsilon}) \subseteq \mathrm{U}_{P'}$, where
$P' \in \triangle'$. By Lemma \ref{t4789}, $\varphi(\Spec K) \subseteq \mathrm{U}_O$.
Obviously $f(\mathrm{U}_O) \subseteq \mathrm{U}_{O'} \subseteq \mathrm{U}_{P'}$.
Then \eqref{E:map1} induces a commutative diagram of rings
\[\xymatrix{R[\Check{P'}] \ar[r]^-{\iota} \ar[d] & A \ar@{ (->}[d] \\
R[G^{\star}] \ar[r]^-{\delta} & K }\]
Put $N = \defset{x \in G^{\star}}{\upsilon(\delta(x)) \geqslant 0}$. If $N=G^{\star}$,
then $\delta(R[G^{\star}])\subseteq A$ and $\delta \colon R[G^{\star}]\to A$ 
induces a morphism $g \colon \Spec A \to \mathrm{U}_O \subseteq X$ which make 
a commutative diagram.
\begin{equation}\label{E:map2}
\vcenter{\xymatrix{\Spec K \ar[r]^-{\alpha} \ar[d] & X \ar[d]^{f}\\
\Spec A \ar[r]^-{\beta} \ar@{-->}[ur]^{g} & X' }}
\end{equation}
So we may assume that $N \neq G^{\star}$. By Lemma \ref{t2313}, $N$ is a concave cone in $G^{\star}$ and $\hat{N}=\mathbb{N}\cdot e$ for some $e \in G$. As $\varphi^{\star}(\Check{P'}) \subseteq N$, we have $e \in \hat{N} \subseteq \varphi^{-1}(P')$. By the assumption, there is a $P \in \triangle$ such that $e \in P \subseteq \varphi^{-1}(P')$. So $\check{P} \subseteq N$, and we have $\delta(R[\check{P}]) \subseteq A$. Hence
\[\delta' \defeq \delta|_{R[\check{P}]} \colon R[\check{P}] \to A\]
induces a morphism $g \colon \Spec A \to \mathrm{U}_P \subseteq X$ which make \eqref{E:map2} commutative. By Valuative Criterion, $f$ is a proper morphism.
\end{proof}

\begin{monoT}
Let $k$ be a field. Let $\varphi\colon(G,\triangle)\to(G',\triangle')$ be a homomorphism
of fans. Then  $\varphi_{\ast} \colon \Temb[k] \to \Temba[k]$ is birational if and only if
$\varphi \colon G \to G'$ is an isomorphism of groups.
\end{monoT}

\begin{proof}
Obviously we have only to prove that if $\varphi_{\ast}$ is birational, then $\varphi$
is an isomorphism. Put $H = G^{\star}$, $H' = G'^{\star}$ and $\psi = \varphi^{\star} \colon H' \to H$.
Since the homomorphism $\phi \colon k[H'] \to k[H]$ induced by $\psi$ is an injective,
so is $\psi$. So we may regard $H'$ as a subgroup of $H$. As $\varphi_{\ast}$ is
birational, $k[H']$ and $k[H]$ have the same quotient field, denoted by $K$. We have
\[\mathrm{rank}(H) = \dim \Temb[k] = \dim \Temba[k] = \mathrm{rank}(H')\ .\]
Hence $H/H'$ is a finite group and $k[H]$ is a finite integral extension of $k[H']$.
By Theorem \ref{kfldsjf}, $k[H']$ is an integral closed integral domain. Hence
$k[H'] = k[H]$, i.e., $H' = H$. Therefore $\varphi \colon G \to G'$ is an isomorphism.
\end{proof}


\begin{thebibliography}{20}
\bibitem{KKato} K. Kato, Toric Singularities, {\em Amer. J. Math.} {\bf 116} (1994), 1073-1099.
\bibitem{MHoch} M. Hochster, Rings of invariants of tori, Chen-Macaulay rings generated by monomials, and polytopes, {\em Ann. of Math.} {\bf 96} (1972), 318-337.
\bibitem{FKato} F. Kato, Log Smooth Deformation Theory, {\em T\`{o}hoku Math. J.} {\bf 48} (1996), 317-354.
\bibitem{TOda} T. Oda, {\em Convex Bodies and Algebraic Geometry}, Springer-Verlag, New York, 1973.
\bibitem{HMat} H. Matsumura, {\em Commutative Algebra}, The Benjamin/Cummings Publishing Company,Inc., 1980, Second Edition.
\bibitem{HMat1} H. Matsumura, {\em Commutative ring theory}, Cambridge University Press, Cambridge, 1986
\bibitem{EGA} Grothendieck, A. and Dieudonn\'{e}, J. {\em El\'{e}ments de G\'{e}om\'{e}trie Alg\'{e}brique.} [EGA] II. \'{E}tude globale \'{e}l\'{e}mentaire de quelques classes de morphismes, Ibid. 8 (1961).
\end{thebibliography}
\end{document}